\documentclass[11pt]{amsart}
\usepackage{amsmath, amssymb, amscd, amsthm, mathrsfs, hyperref}
\usepackage{array}
\usepackage{ifpdf, xspace}
\usepackage{enumitem, color}
\usepackage{fancyhdr}
\usepackage{graphicx}
\usepackage{commath}
\usepackage{mathtools}
\usepackage{bbm}
\usepackage{caption}
\hypersetup{colorlinks=true,linkcolor  = black, citecolor=blue, breaklinks=true}

\newtheorem{theorem}{Theorem}[section]

\newtheorem{lemma}[theorem]{Lemma}

\newtheorem{proposition}[theorem]{Proposition}

\newtheorem{thmz}{Theorem}

\theoremstyle{definition}
\newtheorem{definition}[theorem]{Definition}

\theoremstyle{remark}
\newtheorem{remark}[theorem]{Remark}

\numberwithin{equation}{section}

\newcommand{\cH}{\mathcal{H}}

\newcommand{\B}{\mathbb{B}}

\newcommand{\R}{\mathbb{R}}

\newcommand{\diam}{\operatorname{diam}}

\newcommand{\dist}{\operatorname{dist}}

\newcommand{\Span}{\operatorname{span}}

\newcommand{\aff}{\mathrm{aff}}

\renewcommand{\hat}{\widehat}

\newcommand{\curv}{\operatorname{curv}}

\newcommand{\obeta}{\hat{\beta}}

\newcommand{\res}{\hbox{ {\vrule height .22cm}{\leaders\hrule\hskip.2cm} }} 

\newcommand{\defeq}{\vcentcolon=}

\renewcommand{\dif}{d}

\def\XXint#1#2#3{{\setbox0=\hbox{$#1{#2#3}{\int}$ }
\vcenter{\hbox{$#2#3$ }}\kern-.6\wd0}}

\newcommand{\restr}{\mathbin{\vrule height 1.6ex depth 0pt width
0.13ex \vrule height 0.13ex depth 0pt width 1.3ex}}

\title{Menger curvatures and $C^{1,\alpha}$ rectifiability of measures }
\author{Silvia Ghinassi}
\address{School of Mathematics, Institute for Advanced Study, Princeton, NJ 05840, USA}
\email{ghinassi@math.ias.edu}

\author{Max Goering}
\address{Department of Mathematics, University of Washington, Seattle, WA 98195, USA}
\email{mgoering@uw.edu}

\date{\today}

\thanks{The impetus for the project came while both authors were in residence at the 2018 PCMI program in Harmonic Analysis. The first author was partially supported by Raanan Schul's NSF grant DMS 1361473 and 1763973. The second author was partially supported by Tatiana Toro's NSF grant DMS 1664867 and Hart Smith's NSF grant DMS 1500098. The authors are also grateful to the anonymous referee for their useful and thorough comments throughout the original manuscript.} 


\begin{document}

\begin{abstract}
We further develop the relationship between $\beta$-numbers and discrete curvatures to provide a new proof that under weak density assumptions, finiteness of the pointwise discrete curvature $\curv^{\alpha}_{\mu;2}(x,r)$ at $\mu$-a.e. $x \in \R^{m}$ implies that $\mu$ is $C^{1,\alpha}$ $n$-rectifiable. 
\end{abstract}
\maketitle

\section{Introduction}
In 1990 Peter Jones introduced the $\beta$-numbers as a quantitative tool to provide control of the length of a rectifiable curve and to prove the Analyst's Traveling Salesman Theorem \cite{atst} in the plane. Kate Okikiolu extended the result to one-dimensional objects in $\R^n$ \cite{okikiolu}. 
In order to study the regularity of Ahlfors regular sets and measures of higher dimensions \cite{davidsemmes91, davidsemmes93}, David and Semmes generalized the notion of $\beta$-numbers, see \eqref{e:dbetap}.  This was the beginning of quantitative geometric measure theory and has led to lots of activity around characterizing uniformly rectifiable measures and their connections to the boundedness of a certain class of singular integral operators.

More recently, rectifiable sets and measures have been studied using the quantitative techniques previously used for uniformly rectifiable measures. For instance, $\beta$-numbers can characterize rectifiability of measures, amongst the class of all measures with various density and mass bounds. See for instance \cite{pajot1997conditions, azzamtolsanrect, tolsanrect, naber,badgerschul2017}.

Several other geometric quantities have also proven to be useful in quantifying the regularity of sets and measures. In this paper, we wish to explore how Menger-type curvatures, see Definitions \ref{d:cmc} and \ref{d:mc}, yield information about $C^{1,\alpha}$ $n$-rectifiability, see Definition \ref{d:rect}, of measures. In 1995, Melnikov discovered an identity for the ($1$-dimensional or classical) Menger curvature \cite{melnikov1995analytic} which, in the complex plane, greatly simplified the existing proofs relating rectifiability to the $L^{2}$-boundedness of the Cauchy integral operator \cite{melnikov1995geometric, mattila1996cauchy}. The dream was that a notion of curvature, and similar identity, could be found in higher dimensions to produce simpler proofs demonstrating the equivalence of uniform rectifiability to the $L^{2}$-boundedness of singular integral operators. Alas, in 1999 Farag showed that in higher-dimensions no such identity could exist \cite{farag1999riesz}. Nonetheless, geometric arguments made with non-trivial adaptations from \cite{leger1999menger} have since been used to characterize uniform rectifiability in all dimensions and codimensions in terms of Menger-type curvatures, \cite{lerman2009high, lerman2011high}. A sufficient condition for rectifiability of sets in terms of higher dimensional Menger-type curvatures appears in \cite{meurer2018integral} and was extended to several characterizations of rectifiable measures under suitable density conditions \cite{goering2018characterizations}.

Menger curvatures have also been used to quantify higher regularity (in a topological sense) of surfaces, see, for instance \cite{strzelecki2011integral, blatt2012sharp}. Of particular relevance in our context, \cite{kolasinski2017higher} showed that finiteness of $\curv_{\mu;p}^{\alpha}$, see Definition \ref{d:mc}, is a sufficient condition for $C^{1,\alpha}$ $n$-rectifiability of measures. A formulation of the theorem is\footnote{The familiar reader may be aware that there are two additional parameters in the theorem of \cite{kolasinski2017higher}. One such parameter is the ability to choose from a small family functions to replace the $h_{\min}$ in the integrand that defines $\curv_{\mu;p}^{\alpha}$. Such choices have previously been shown to be comparable to one another, see \cite{lerman2009high, lerman2011high, lerman2012least} or \cite[8.6-8.8]{kolasinski2017higher}. The second such parameter was originally denoted by $l$. The $l = n+2$ case is written here. Since any other choice of $l$ is a stronger assumption, (changing the parameter $\ell$ is equivalent to replacing an $L^{p}$ bound on some number of components of the integrand with an $L^{\infty}$ bound) we chose to remove this for readability.}

\begin{theorem} \label{t:kol}
Let $\mu$ be a Radon measure on $\R^{m}$, with $0 < \Theta^{n}_{*}(\mu,x) \le \Theta^{n,*}(\mu,x) < \infty$, for $\mu$-a.e. $x \in \R^{m}$ and let $1 \le p < \infty$, $0 < \alpha \le 1$. If for $\mu$- a.e. $x \in \R^{m}$
\begin{equation} \label{e:fmpc}
\curv_{\mu;p}^{\alpha}(x,1) < \infty,
\end{equation}
then $\mu$ is $C^{1,\alpha}$ $n$-rectifiable.
\end{theorem}

The goal of this article is to prove that for a Radon measure $\mu$ satisfying relaxed density assumptions and for any $\alpha \in [0,1)$ if the pointwise Menger curvature with $p=2$, see Definition \ref{d:mc}, is finite at $\mu$ a.e. $x$ then $\mu$ is $C^{1,\alpha}$ $n$-rectifiable. More precisely,

\begin{thmz} \label{t:maintheorem}
Let $\mu$ be  a Radon measure on $\R^{m}$, with $0<\Theta^{n,*}(\mu,x)< \infty$ for $\mu$-a.e. $x \in \R^{m}$, and let $\alpha \in [0,1)$. 
If for $\mu$- a.e. $x \in \R^{m}$
\begin{equation} \label{e:fmc}
\curv^{\alpha}_{\mu;2}(x,1) < \infty,
\end{equation}
then  $\mu$ is $C^{1,\alpha}$ $n$-rectifiable.
\end{thmz}

The $\alpha = 0$ case in Theorem \ref{t:maintheorem} appears in \cite{goering2018characterizations}. The case $\alpha >0$ is an improvement of a special case of \cite{kolasinski2017higher} where the lower density assumption is relaxed.

A rough sketch of the proof is as follows: the condition \eqref{e:fmc} is shown to imply ``flatness'' of the support of $\mu$ in terms of Jones' square function, and consequently (pieces of) the support of $\mu$ can be parametrized by Lipschitz graphs (see \cite[Theorem 5.4]{meurer2018integral}) when $\alpha = 0$ or $C^{1,\alpha}$ images (see \cite[Theorem II]{ghinassi2017sufficient}) when $\alpha > 0$.

On the other hand, the original proof provided by Kolasi{\'n}ski had two major parts. First, under the additional assumption that $\mu$ is Lipschitz $n$-rectifiable Kolasi{\'n}ski showed that the condition \eqref{e:fmpc} forced additional flatness on the Lipschitz functions that cover the support of $\mu$, which consequently implied better regularity on each such function (see \cite[Lemma A.1]{schatzle2009lower}). Second, it remained to show that \eqref{e:fmpc} implies Lipschitz $n$-rectifiability of $\mu$. This was done by appealing to a characterization of rectifiability from \cite[\S 2.8 Theorem 5]{allard1972first} which roughly says that if the approximate tangent cone (in the sense of Federer) of $\mu$ is contained in an $n$-plane at almost every point, then $\mu$ is Lipschitz $n$-rectifiable.

\begin{remark} \label{r:kp2}
Note that Theorem \ref{t:maintheorem} requires $0 < \Theta^{n,*}(\mu,x)$ for $\mu$ a.e. $x$, whereas Theeorem \ref{t:kol} requires the stronger assumption that $0 < \Theta^{n}_{*}(\mu,x)$ for $\mu$ a.e. $x$. The stronger density assumption in Theorem \ref{t:kol} allows one to apply Theorem \ref{t:ARcase}. Then Remark \ref{r:mp} and Proposition \ref{p:betap} provide a direct proof that
$$
\int_{0}^{1} \beta^{\mu}_{2}(x,r)^{2} \frac{dr}{r} < \infty \quad \mu ~ a.e. ~ x \in \R^{m}
$$
when working under the hypotheses of Theorem \ref{t:kol}. This provides an alternative proof to what we previously called the second major part of the original proof of Theorem \ref{t:kol}.

The question of whether the proof of Theorem \ref{t:maintheorem} when $\alpha = 0$ can be completed by appealing to \cite{azzamtolsanrect} after controlling Jones' function as above is an interesting one. Presently, the authors do not know how to do this without additionally assuming $0 < \Theta^{n}_{*}(\mu,x)$ for $\mu$ almost every $x$.

Another difference between the two theorems is that Theorem \ref{t:maintheorem} is stated only when $p=2$. Since increasing $p$ only makes condition \eqref{e:fmpc} harder to satisfy, the results in Theorem \ref{t:kol} are not sharp in terms of the parameter $p$. We expect that varying the parameter $p$ would lead to results about rectifiability in the sense of Besov spaces, which is beyond the scope of this article.
\end{remark}
The proof of Theorem \ref{t:maintheorem} is divided in two parts. First, we prove the claim for a measure $\mu$ which is $n$-Ahlfors upper regular on $\R^m$ and with positive lower density. Then, we use standard techniques to reduce the general case to the previous one.

\section{Notation and Background}

We begin by stating some definitions and notation.
\begin{definition}
Let $0 \leq s < \infty$ and let $\mu$ be a measure on $\R^m$. The upper and lower $s$-densities of $\mu$ at $x$ are defined by
\begin{align*} 
\Theta^{s,*}(\mu,x)& =\limsup_{r \to 0}\frac{\mu(B(x,r))}{r^s} \\
\Theta_{*}^s(\mu,x) &=\liminf_{r \to 0}\frac{\mu(B(x,r))}{r^s}.
\end{align*}
If they agree, their common value is called the $s$-density of $\mu$ at $x$ and denoted by
\begin{equation}
    \Theta^s(\mu,x)=\Theta^{s,*}(\mu,x)=\Theta_{*}^s(\mu,x).
\end{equation}
\end{definition}

\begin{definition}[\emph{$\beta_{p}$-numbers}] 
Given $x \in \R^m$, $r>0$, $p \in (1,\infty)$, an integer $0\leq n \leq m$, and a Borel measure $\mu$ on $\R^{m}$ define 
\begin{equation} \label{e:dbetap}
\beta^{\mu}_{p}(B(x,r))=\left(\inf_{L}  \frac{1}{r^n} \int_{B(x,r)} \left(\frac{\dist(y,L)}{r}\right)^p d\,\mu(y)\right)^{\frac{1}{p}},
\end{equation}
where the infimum is taken over all $n$-planes $L$.
 \end{definition}
 
When we talk about rectifiability and higher-order rectifiability, we mean what Federer would call ``countably $n$-rectifiable''. 

\begin{definition} \label{d:rect}
A measure $\mu$ on $\R^n$ is said to be (Lipschitz) $n$-rectifiable if there exist countably many Lipschitz maps $f_i \colon \R^n \to \R^m$  such that
\begin{equation} \label{e:drect}
\mu \left( \R^m \setminus \bigcup_i f_i(\R^n)\right)=0.
\end{equation}

A measure $\mu$ on $\R^n$ is $C^{1,\alpha}$ $n$-rectifiable if there exist countably many $C^{1,\alpha}$ maps $f_i \colon \R^n \to \R^m$  such that \eqref{e:drect} holds.
\end{definition}

In \cite{ghinassi2017sufficient}, a sufficient condition for $C^{1,\alpha}$ $n$-rectifiability in terms of $\beta$-numbers is provided.

 \begin{theorem}[\cite{ghinassi2017sufficient}] \label{t:silvia}
Let $\mu$ be a Radon measure on $\R^{m}$ such that \linebreak $\Theta^n_*(\mu,x)< \infty$ and $\Theta^{n,*}(\mu,x)>0$ for $\mu$-almost every $x \in \R^{m}$, and $\alpha \in (0,1)$. Moreover, suppose, for $\mu$-almost every $x \in \R^{m}$,
\begin{equation} \label{e:rect1}
J^{\mu}_{2,\alpha}(x) \defeq \int_0^1 \frac{\beta_2^{\mu}(x,r)^2}{r^{2\alpha}} \,\frac{dr}{r} <\infty
\end{equation}
Then, $\mu$ is $C^{1,\alpha}$ $n$-rectifiable.

When $\alpha=1$, if we replace $r$ in the left hand side of (\ref{e:rect1}) by $r\eta(r)$, where $\eta(r)^2$ satisfies the Dini condition, then we obtain that $\mu$ is $C^{2}$ $n$-rectifiable.
\end{theorem}
\begin{remark} 
We say that a function $\omega$ satisfies the Dini condition if \linebreak $\int_{0}^{1} \frac{\omega(r)}{r} dr < \infty$. A possible choice for $\eta$ in Theorem \ref{t:silvia} is $\eta(r)=\frac{1}{\log(1/r)^{\gamma}}$, for $\gamma > \frac12$.
\end{remark}

\begin{definition}[\emph{Classical Menger curvature}] \label{d:cmc}
Given three points $x,y,z \in \R^{m}$, the (classical) Menger curvature is defined to be the reciprocal of the circumradius of $x,y,z$. That is,
$$
c(x,y,z) = \frac{1}{R(x,y,z)},
$$
where $R(x,y,z)$ is the radius of the unique circle passing through $x,y,z$.
\end{definition}

In order to work with higher dimensional Menger curvatures, we introduce some notation for simplices in $\R^m$.
\begin{definition}[\emph{Simplices}]
Given points $\{x_{0}, \dots, x_{n}\} \subset \R^{m}$, $\Delta(x_{0}, \dots, x_{n})$ will denote the convex hull of $\{x_{0}, \dots, x_{n}\}$. In particular, if $\{x_{0}, \dots, x_{n}\}$ are not contained in any $(n-1)$-dimensional plane, then  $\Delta(x_{0}, \dots, x_{n})$ is an $n$-dimensional simplex with corners $\{x_{0}, \dots, x_{n}\}$. Moreover, we denote by $\aff\{x_{0}, \dots, x_{n}\}$ the smallest affine subspace containing $\{x_{0}, \dots, x_{n}\}$. That is $\aff \{x_{0}, \dots, x_{n} \} = x_{0} + \Span \{x_{1}- x_{0}, \dots, x_{n} - x_{0} \}$. Then, we define
\[
h_{\min}(x_{0}, \dots, x_{n}) = \min_{i} \dist(x_{i}, \aff\{x_{0}, \dots, x_{i-1}, x_{i+1}, \dots, x_{n} \}),
\]
to be the minimum height of a vertex over the plane spanned by the opposing face. If $\Delta = \Delta(x_{0}, \dots, x_{n})$ we occasionally abuse notation and write $h_{\min}(\Delta)$ in place of $h_{\min}(x_{0}, \dots, x_{n})$. If $\Delta$ as before is an $n$-simplex, it is additionally called an $(n, \rho)$-simplex if 
\[
h_{\min}(x_{0}, \dots, x_{n}) \ge \rho.
\]
\end{definition}

\begin{definition}[\emph{Menger curvatures}] \label{d:mc}
For $x \in \R^m$, $r>0$, $\alpha \in [0,1)$, an integer $0\leq n \leq m$, and $p \in [1,\infty]$, we define the curvature of $\mu$ at $x$ of scale $r$ to be 
\begin{align} \label{e:silcurv} \hspace{.75in}
\hspace{-.8in}\curv^{\alpha}_{\mu;p}(x,r) = \int_{B(x,r)^{n+1}} \frac{h_{\min}(x,x_{1}, \dots, x_{n+1})^p}{\diam \left( \{x,x_{1}, \dots, x_{n+1}\} \right)^{p(1+\alpha) + n(n+1)}} d\mu^{n+1},
\end{align}
where $\mu^{n+1}$ is the product measure defined by taking $(n+1)$-products of $\mu$ with itself.
\end{definition}

\section{Proof of Theorem \ref{t:maintheorem}} \label{s:AR}

We now proceed to prove the theorem in the case where $\mu$ is $n$-Ahlfors upper regular on $\R^m$ and has positive lower density $\mu$-almost everywhere.

We recall the following Lemma from \cite[Lemma 3.13]{goering2018characterizations}, which says that, under the appropriate density assumptions on a measure $\mu$, given a point $x$ and radius $r$, the ball $B(x, r)$ contains a large number of effective $n$-dimensional secant planes through $x$. 
\begin{lemma} \label{t:likeNV}
Let $\mu$ be an $n$-Ahlfors upper-regular Radon measure on $\R^{m}$ with upper-regularity constant $C_{0}$. Suppose $x \in \R^{m}$ and $\lambda, R > 0$ such that 
\begin{equation} \label{E:1}
\mu(B(x,r)) \ge \lambda r^{n}
\end{equation} 
holds for all $0 < r \le R$.

Then, for
\begin{equation} \label{E:deltadef}
\delta = \delta(n, \lambda, C_{0}) = \frac{\lambda}{2^{k+2} 5^{n-1} C_{0}} 
\end{equation}
and 
\begin{equation} \label{E:etadef}
\eta = \eta(n, \lambda, C_{0}) = \frac{\delta}{10n} = \frac{\lambda}{2^{k+3} 5^{n} n C_{0}}
\end{equation}
and all $0 < r \le R$ there exist points $\{x_{i,r}\}_{i=1}^{n}  \subset B(x,r)$ such that
\begin{equation} \label{E:bigh}
h_{\min}(x, x_{1,r}, \dots, x_{n,r}) \ge \delta r
\end{equation}
and
\begin{equation} \label{E:0}
(\mu \restr B(x,r)) (B(x_{i,r}, 5 \eta r)) \ge  \left( \frac{\lambda \eta^{m}}{2^{m+1}} \right) r^{n} = C_{2}(m,n,\lambda, C_{0}) r^{n}.
\end{equation}
In particular, if for each $i \in \{1, \dots, n\}$, $B_{i,r} \defeq B(x_{i,r}, 5 \eta r)$ for any choices of $y_{i} \in B_{i,r}$ it follows that
\begin{equation}\label{E:fczfat}
h_{\min}(x,y_{1}, \dots, y_{n}) \ge \delta r - 5 n \eta r = \frac{\delta r}{2}.
\end{equation} 

Finally, if $\mathbb{B}_{r} \defeq B_{1,r} \times \dots \times B_{n,r}$ then 
\begin{equation} 
\label{e:disjoint} \mathbb{B}_{\frac{\delta r}{3}} \cap \mathbb{B}_{r} = \emptyset.
\end{equation}
\end{lemma}

\begin{theorem} \label{t:ARcase}
If $\mu$ is an $n$-Ahlfors upper-regular measure on $\R^{m}$ such that $\Theta^{n}_{*}(\mu,x) > 0$ for all $x$, and
\[
\curv^{\alpha}_{\mu;2}(x,1) < \infty
\]
for $\mu$-almost every $x$, then $\mu$ is $C^{1,\alpha}$ $n$-rectifiable.
\end{theorem}

We will prove Theorem \ref{t:ARcase} by showing that for $\mu$ as in the theorem statement, we can in fact show that
\[
\int_{0}^{1} \frac{\beta^{\mu}_{2}(x,r)^{2}}{r^{2\alpha}} \,\frac{\dif r}{r} < \infty
\]
for almost every $x \in \R^{m}$, and then appeal to Theorem \ref{t:silvia}. 
In the proof we will use a slight modification of the usual $\beta$-numbers introduced above, the so-called ``centered $\beta$-numbers'', that we denote by $\obeta^{\mu}_{2}$. These numbers are defined exactly as the $\beta$-numbers, except that the infimum for $\obeta^{\mu}_{2}(x,r)$ is restricted to $n$-planes passing through $x$. That is, for $x \in \R^m$, $r>0$ define
\[
\obeta^{\mu}_{2} (x,r)^{2} \defeq \inf_{L \ni x} \int_{B(x,r)}\left(\frac{\dist(z, L)}{r}\right)^{2}\,\frac{d\mu(z)}{r^n}.
\]
In particular, because the infimum in the definition of $\beta_2^{\mu}(x,r)$ is taken over a larger class than the one in $\obeta^{\mu}_{2} (x,r)$, we have $\beta^{\mu}_{2}(x,r) \le \obeta^{\mu}_{2}(x,r)$ for all $x \in \R^{m}, r > 0$.

\begin{proof}[Proof of Theorem \ref{t:ARcase}]
Let $\mu$ be as in the theorem statement, and $x$ a point so that $\Theta^{n}_{*}(\mu,x) > 0$. Then, there exists some $\lambda > 0$ so that for all $0 < r \le 1$, $\mu(B(x,r)) \ge \lambda r^{n}$. Now, fix $0 < r \le 1$. By the definition of infimum, for any $(y_1, \dots, y_n) \in (\R^{n})^{n}$, 
\begin{align}
\nonumber \obeta^{n}_{\mu;2}&(x,r)^{2}  = \inf_{L \ni x} \frac{1}{r^{n}} \int_{B(x,r)} \left( \frac{\dist(z, L)}{r} \right)^{2} d \mu(z) \\
 \label{e:30} & \leq \frac{1}{r^{n}} \int_{B(x,r)} \left( \frac{\dist(z, \aff\{x,y_{1}, \dots, y_{n}\})}{r} \right)^{2} d \mu(z).
 \end{align}
Choose $\{x_{i,r}\}, B_{i,r},$ and $\mathbb{B}_{r}$ as in Lemma \ref{t:likeNV}. Averaging \eqref{e:30} over all $(y_{1}, \dots, y_{n}) \in \B_r$, and then applying \eqref{E:0} yields
\begin{align}
\nonumber \obeta^{n}_{\mu;2}&(x,r)^{2} \leq \int_{\mathbb{B}_{r}} \int_{B(x,r)}  \left(\frac{\dist(z,\aff\{x,y_{1}, \dots, y_{n}\})}{r}\right)^{2} \frac{ d\mu(z) d \mu^{n}(y_{1}, \dots, y_{n})}{\mu^{n}(\mathbb{B}_{r}) r^{n}} \\
\label{e:34} & \le C \int_{\mathbb{B}_{r}} \int_{B(x,r)} \left(\frac{ \dist(z, \aff \{x,y_{1}, \dots, y_{n}\})}{r} \right)^{2} \frac{d \mu(z) d\mu^{n}(y_{1}, \dots, y_{n})}{r^{n^{2} + n}},
\end{align}
where $C= C(m,n,\lambda, C_{0})$. 
%
We now claim that,
\begin{equation} \label{e:36}
\dist(z, \aff \{x,y_{1}, \dots, y_{n}\}) \le \left( \frac{2}{\delta} \right)^{n} h_{\min}(x,z,y_{1}, \dots, y_{n}).
\end{equation}

Indeed, let $\Delta = \Delta(x,z,y_{1}, \dots, y_{n})$ and $\Delta_{w} = \Delta ( \{x,z, y_{1}, \dots, y_{n}\} \setminus \{w\})$, for each $w \in \{x,z,y_{1}, \dots, y_{n})$. Basic Euclidean geometry ensures that
\begin{align}  \label{e:37}
\dist(z, \aff\{ \Delta_{z} \}) \cH^{n}(\Delta_{z}) &=  (n+1) \cH^{n+1}(\Delta) \\
\nonumber & = h_{\min}(x,z,y_{1}, \dots, y_{n}) \cH^{n}(\Delta_{w_{0}}),
\end{align}
where $w_{0}$ is any vertex such that 
\[
\dist(w_{0}, \aff\{ \{x,z,y_{1}, \dots, y_{n}\} \setminus \{w_{0}\}\}) = h_{\min}(x,z,y_{1}, \dots, y_{n}).
\]
 On the other hand, since $\{x,z\} \cup B_{i,r} \subset B(x,r)$ for all $i = 1, \dots, n$, Equation \eqref{E:fczfat} ensures
\begin{equation} \label{e:38}
\frac{\cH^{n}(\Delta_{w_{0}})}{\cH^{n}(\Delta_{z})} \le \left( \frac{2r}{\delta r} \right)^{n}.
\end{equation}
The claim \eqref{e:36} now follows from \eqref{e:37} and \eqref{e:38}.

Evidently, $\diam\{x,z,y_{1}, \dots, y_{n}\} \le 2r $. Using this diameter bound, \eqref{e:36} and \eqref{e:34}, we conclude
\begin{align} \label{e:39}
\obeta^{2}_{\mu}(x,r)^{2} \le C \int_{\mathbb{B}_{r}} \int_{B(x,r)} \frac{h_{\min}(x,z,y_{1}, \dots, y_{n})^{2}}{\diam\{x,z,y_{1}, \dots, y_{n}\}^{n^{2} + n + 2}} d \mu(z) d \mu^{n}(y_{1}, \dots, y_{n}).
\end{align}

Setting $r_j = \left(\frac{\delta}{3}\right)^j$ and using the fact that $0 < \frac{\delta}{3} < 1$, it's known
\begin{equation} \label{e:40}
\int_{0}^{1} \frac{\obeta^{\mu}_{2}(x,r)^{2}}{r^{2\alpha}} \frac{dr}{r} \le C_{\delta} \sum_{j \ge 0} \frac{\obeta^{\mu}_{2}(x, r_j)^{2}}{(r_j)^{2\alpha}}.
\end{equation}
It now follows from \eqref{e:40}, \eqref{e:disjoint}, and \eqref{e:39} that
\begin{align*}
& \int_{0}^{1} \obeta^{\mu}_{2}(x,r)^{2} \frac{dr}{r^{1+2\alpha}}  \\
& \le C \int_{\cup_{j} \mathbb{B}_{r_{j}} \times B(x,r_{j})} \frac{h_{\min}(x,z,y_{1}, \dots, y_{n})^{2}}{\diam\{x,z,y_{1}, \dots, y_{n}\}^{n^{2} + n + 2 }r_{j}^{2\alpha}} d \mu^{n+1}(y_{1}, \dots, y_{n},z) \\
& \le  C \int_{\cup_{j} \mathbb{B}_{r_{j}} \times B(x,r_{j})} \frac{h_{\min}(x,z,y_{1}, \dots, y_{n})^{2}}{\diam\{x,z,y_{1}, \dots, y_{n}\}^{n^{2} + n + 2 + 2 \alpha }} d \mu^{n+1}(y_{1}, \dots, y_{n},z) \\
& \leq C \int_{B(x,1)^{n+1}} \frac{h_{\min}(x,x_{1}, \dots, x_{n+1})^2}{\diam \left( \{x,x_{1}, \dots, x_{n+1}\} \right)^{2(1+\alpha) + n(n+1)}} d\mu^{n+1}(x_1, \dots, x_{n+1}) \\
& = C \curv^{\alpha}_{\mu;2}(x,1),
\end{align*}
where in the penultimate step we used that $\mathbb{B}_{r_{j}} \times B(x,r_{j}) \subset B(x,1)^{n+1}$ for all $j$, and the non-negativity of the integrand.
\end{proof}

\begin{remark} \label{r:mp}
At this point, we briefly focus on the difference between conditions \eqref{e:fmc} and \eqref{e:fmpc}. The proof of Theorem \ref{t:ARcase} could be followed out identically, using $\curv_{\mu;p}^{\alpha}(x,1) < \infty$ in place of $\curv_{\mu;2}^{\alpha}(x,1)$ and by replacing appropriate $2$'s with $p$'s,  to obtain
$$
\int_{0}^{R} \left( \frac{\obeta^{\mu}_{p}(x,r)}{r^{\alpha}} \right)^{p} \frac{dr}{r} \le C \curv_{\mu;p}^{\alpha}(x,R) < \infty.
$$
Consequently, the following proposition is of interest, see Remark \ref{r:kp2}.
\end{remark}

\begin{proposition} \label{p:betap}
Let $\mu$ be a Radon measure on $\R^m$ such that $0<\Theta^{n,*}(\mu,x)$ and $\mu$ is Ahlfors upper-regular with constant $C_{0}$ for $\mu$-almost every $x \in \R^{m}$, $p \in [1,\infty)$, and $\alpha \in (0,1]$. If for $\mu$-a.e. $x \in \R^{m}$,
\[
\int_0^1 \left(\frac{\beta_p(x,r)}{r^{\alpha}}\right)^p \, \frac{dr}{r} < \infty,
\]
then $\mu$ is $n$-rectifiable.
\end{proposition}
\begin{proof}
We show that the hypotheses imply that for $\mu$- a.e. $x \in \R^{m}$,
\[
\int_0^1 \beta^{\mu}_2(x,r)^2 \, \frac{dr}{r} < \infty
\]
and then employ \cite{naber} to obtain rectifiability.

For $p \leq 2$, it is enough to observe, from the definition of $\beta_p(x,r)^p$, that 
\[
\left(\frac{\dist(y,P)}{r}\right)^2 \leq 2^{\frac{2}{p}} \left(\frac{\dist(y,P)}{r}\right)^p
\]
as $\frac{\dist(y,P)}{r} \leq 2$. This immediately implies that $\beta_2(x,r)^2 \leq \beta_p(x,r)^p$, and hence we are done. Note that in this case we did not use that $\alpha>0$.

For $p>2 $, we use H\"older inquality:
\begin{align*}
\int_0^1 \beta_2(x,r)^2 \, \frac{dr}{r}& = \int_0^1\frac{ \beta_2(x,r)^2}{r^{2\alpha}} \cdot r^{2\alpha} \, \frac{dr}{r} \\
& \leq \left( \int_0^1\left(\frac{ \beta_2(x,r)^2}{r^{2\alpha}}\right)^{\frac{p}{2}}  \, \frac{dr}{r}\right)^{\frac{2}{p}}  \left( \int_0^1 (r^{2\alpha})^{\frac{p}{p-2}} \, \frac{dr}{r}\right)^{\frac{p-2}{p}}  \\
& \leq \left(  \int_0^1\left(\frac{ \beta_2(x,r)}{r^{\alpha}}\right)^{p}  \, \frac{dr}{r}\right)^{\frac{2}{p}}  \left( \int_0^1 r^{\frac{2p\alpha}{p-2}-1} \, dr\right)^{\frac{p-2}{p}}  \\
& \leq C_{p,\alpha,C_{0}}  \left(\int_0^1\left(\frac{ \beta_p(x,r)}{r^{\alpha}}\right)^{p}  \, \frac{dr}{r}\right)^{\frac{2}{p}},
\end{align*}
where in the last inequality we used the fact that, for $p>2$, \linebreak $\beta_2(x,r) \leq \left( \frac{ \mu(B(x,r))}{r^{n}} \right)^{\frac12-\frac{1}{p}} \beta_p(x,r)$.
\end{proof}

Finally, we reduce Theorem \ref{t:maintheorem} to Theorem \ref{t:ARcase}. 

\begin{lemma} \label{l:chopping}
Let $\mu$ be a Radon measure on $\R^m$ such that $0<\Theta^n_*(\mu,x) \leq \Theta^{n,*}(\mu,x) < \infty$ for $\mu$-a.e. $x \in \R^m$. Then there exist measures $\mu_k$ such that for every set $A$,
\[
\mu(A) = \lim_{k \to \infty} \mu_k(A)=\bigcup_{k=1}^{\infty}\mu_k(A),
\]
and such that each $\mu_k$ is upper $n$-Ahlfors regular and $\Theta^n_*(\mu_k,x)>0$ for $\mu_k$-a.e. $x$.
\end{lemma}

In Lemma \ref{l:chopping} we assume that $0 < \Theta^{n}_{*}(\mu,x)$ for $\mu$-a.e. $x \in \R^m$. This follows from the hypotheses of Theorem \ref{t:maintheorem}. In fact, if $\mu$ satisfies the hypothesis of Theorem \ref{t:maintheorem} with $\alpha = 0$, then $\mu$ is $n$-rectifiable (see \cite[Theorem 1.19]{goering2018characterizations}). Moreover, for $\alpha \in (0,1)$, $\curv^{\alpha}_{\mu;2}(x,1) < \infty$ implies $\curv^{0}_{\mu;2}(x,1) < \infty$. Now it is enough to observe that if a measure $\mu$ is $n$-rectifiable, then $\Theta^{n}_{*}(\mu,x)>0$, for $\mu$-a.e. $x \in \R^m$.

\begin{proof} 
For any positive integer $k$, let $\mu_k$ be defined by $\mu_k =\mu\res_{E_K}$, where $E_k$ is given by
\[
E_k=\{x \in \R^m \mid \mu(B(x,r)) \leq k r^n, \text{for every $r<2^{-k}$}\}.
\]
Clearly $\mu(A)=\lim_{k \to \infty} \mu(E_k \cap A)$.

Notice that $\mu_k(B(x,r))\leq \mu(B(x,r)) \leq k r^n$ so that $\mu_k$ is upper $n$-Ahlfors regular. Moreover $\Theta^n_*(\mu_k,x)>0$ for almost every $x \in E_k$, by Theorem 2.12(2) in \cite{mattila}, since $\mu_k \ll \mu$ and $\Theta^n_*(\mu,x)>0$.  
\end{proof}

\begin{proof}[Proof of Theorem \ref{t:maintheorem}]

Now, let $\mu_k$ be as in Lemma \ref{l:chopping}. Then we can apply Theorem \ref{t:ARcase} to each $\mu_k$.
Because each $\mu_k$ is $C^{1,\alpha}$ $n$-rectifiable, it follows that $\mu$ is $C^{1,\alpha}$ $n$-rectifiable.
\end{proof}



\end{document}